\documentclass{article}

\usepackage[mathscr]{euscript}
\usepackage{pstricks} 
\usepackage{graphicx, pst-plot, pst-math} 
\usepackage{amsfonts, dsfont, ulem,csquotes}
\usepackage{amsmath, latexsym}
\usepackage{amssymb,verbatim}
\usepackage{theorem}
\usepackage{pgf,tikz,pgfplots}
\usetikzlibrary{arrows}

\usepackage{hyperref, ulem}
\usepackage{graphicx}
\usepackage{xcolor}

\newtheorem{lem}  {Lemma}
\newtheorem{pro}[lem]    {Proposition}
\newtheorem{thm}[lem]    {Theorem}
\newtheorem{cor}[lem]    {Corollary}

\theorembodyfont{\rmfamily}
\newtheorem{rem}{Remark}
\newtheorem{exa}     {Example}

\newcommand {\den} {{e}_\eps}

\newcommand{\R}{\mathbb{R}}

\newcommand{\eps}{\varepsilon}
\newcommand{\e}{\varepsilon}

\newcommand{\h}{{\mathcal{H}}}
\newcommand{\proof}[1]{\par\medskip\noindent{\bf Proof#1.}}
\newcommand{\qed}{\hfill$\square$}
\newcommand{\di}{\nabla \cdot}
\newcommand{\be}{\begin{equation}}
\newcommand{\ee}{\end{equation}}

\newcommand{\nd}{\noindent}

\newcommand{\Om}{\Omega}
\newcommand{\dOm}{\partial\Omega}
\newcommand{\dist}{\mathop{\rm dist \,}}

\newcommand{\f}{\varphi}

\newcommand{\degr}{\operatorname{deg}}

\newcommand{\norm}[1]{\left\Vert #1 \right\Vert}
\newcommand{\abs}[1]{\left\vert #1\right\vert}

\newcommand {\bun}{\alpha_1}
\newcommand {\bd}{\alpha_2}
\newcommand {\cd}{\mathcal{D}}

\title{Global uniform estimate for the modulus of $2D$ Ginzburg-Landau  vortexless solutions with asymptotically infinite boundary energy}
\author{Radu Ignat\footnote{Institut de Math\'ematiques de Toulouse \& Institut Universitaire de France, UMR 5219, Universit\'e de Toulouse, CNRS, UPS
IMT, F-31062 Toulouse Cedex 9, France. Email: Radu.Ignat@math.univ-toulouse.fr} 
\and  Matthias Kurzke\footnote{School of Mathematical Sciences, University of Nottingham, Nottingham NG7 2RD, United Kingdom. Email: matthias.kurzke@nottingham.ac.uk}
 \and Xavier Lamy \footnote{Institut de Math\'ematiques de Toulouse, UMR 5219, Universit\'e de Toulouse, CNRS, UPS
IMT, F-31062 Toulouse Cedex 9, France. Email: Xavier.Lamy@math.univ-toulouse.fr}}

\begin{document}
\maketitle

\begin{abstract}
For $\eps>0$, let $u_\eps:\Omega\to \R^2$ be a solution of the Ginzburg-Landau system $$-\Delta u_\eps=\frac1{\eps^2} u_\eps (1-|u_\eps|^2)$$ in a Lipschitz bounded domain $\Omega$. 
 In an energy regime that excludes interior vortices,
we prove that $1-|u_\eps|$ is uniformly estimated by a positive power of $\eps$ {\bf globally} in $\Omega$ provided that the energy of $u_\eps$ at the boundary $\partial \Om$ does not grow faster
 than $\eps^{-\alpha}$ with $\alpha\in (0,1)$.
\end{abstract}

\section{Introduction}
Let $\Omega\subset \R^2$ be a Lipschitz bounded open connected set (not necessarily simply connected) with the unit outer normal  and tangent vector fields $(\nu, \tau)$ defined a.e. on $\partial \Omega$ with $$\tau=\nu^\perp=(-\nu_2, \nu_1)$$ so that $(\nu, \tau)$ forms an oriented frame a.e. on $\dOm$. For every small $\eps>0$, 
let $u_\eps:\Omega \to \R^2$ be a solution of the Ginzburg-Landau system:
\be
\label{eq:euler}
\left\lbrace
\begin{aligned}
-\Delta u_\eps&=\frac1{\eps^2} u_\eps (1-|u_\eps|^2) & \textrm{in } \Omega,\\
u_\eps&=g_\eps & \textrm{on } \partial \Omega
\end{aligned}
\right.
\ee
with the boundary data
$g_\eps:\partial \Omega \to \R^2$. 
 For the boundary energy
\be
\label{bdry}
N_\eps:=\int_{\partial \Omega} \frac12|\partial_\tau g_\eps|^2 +\frac1{4\eps^2} (1-|g_\eps|^2)^2\, d\h^1
\ee
and  the interior energy
\be
\label{int}
M_\eps:=\int_{\Omega} \frac12|\nabla u_\eps|^2 + \frac1{4\eps^2}(1-|u_\eps|^2)^2\, dx,\ee
we assume that there exists a power $\alpha\in (0,1)$ such that \footnote{We  write $a\ll b$ if $\frac a b\to 0$, and $a\lesssim b$ is $\frac ab$ is bounded by a universal constant.}
\be
\label{regime}
M_\eps\le  \kappa |\log \eps| \quad \textrm{and} \quad N_\eps\ll \frac 1{\eps^\alpha} \quad \textrm{as } \, \eps\to 0,
\ee
 for some small constant $\kappa>0$ depending on the Lipschitz regularity of $\Omega$.
The first condition in \eqref{regime} avoids nucleation of interior vortices of non-vanishing winding number (because the energetic cost of an interior vortex of non-zero winding number 
is of order $|\log \eps|$, see the seminal book of Bethuel-Brezis-H\'elein \cite{BethuelBrezisHelein:1994a}). The second condition in \eqref{regime} corresponds to an energetic regime avoiding
 the presence of boundary vortices: indeed, a transition of $g_\eps$ between two opposite directions at the boundary on a distance $\eps$ 
(the length scale of a vortex) has an energetic cost of order $\frac1\eps$ (see Example \ref{ex1} below). If $N_\e\lesssim \frac 1\e$, then solutions $u_\eps$ of \eqref{eq:euler} may have zeros on the boundary (see Proposition \ref{p:optimal2}). 
 
\subsection{Main result}

Our main result is the following global uniform estimate in the regime \eqref{regime} for the convergence of $\abs{u_\e}$ to $1$ in $\Omega$, which means that $1-|u_\eps|$ behaves as a positive power of $\eps$.

\begin{thm}\label{thm:main}
Let $\Omega\subset \R^2$ be a Lipschitz bounded domain. There exists a (small) constant $\kappa>0$ depending on the Lipschitz regularity of $\Omega$ such that 
for every solution $u_\e$ of \eqref{eq:euler} satisfying \eqref{regime}  for some $\alpha\in (0,1)$
 we have the following global estimate \footnote{We denote by $a+$ (resp. $a-$) any number strictly bigger than $a$ (resp. strictly smaller than $a$) that one can think of as close to $a$.  The constants 
in inequalities involving $a+$ or $a-$ may depend on the choice of these numbers.}
 $$\sup_{\Omega}\big|{|u_\eps|}-1\big|\leq C \bigg(\eps^{1-}(1+N_\eps+ M_\eps)(1+M_{\eps})^{\frac12-}\bigg)^{\frac16-} \quad \textrm{ as } \quad \eps\to 0,$$
 for some constant $C>0$ depending only on the Lipschitz regularity \footnote{In fact, $C>0$ depends only on the lowest angle and lowest radius of interior and exterior cones at any point of the Lipschitz  boundary $\partial \Om$.  } of $\Omega$. In particular, $g_\eps$ has zero winding number on $\dOm$, i.e.,\footnote{In general, $\dOm$ is not connected; the definition of the degree is coherent with the choice of the orientation $\tau=\nu^\perp$ given by the outer normal field $\nu$.}
$$\degr(g_\eps, \dOm):=\frac1{2\pi} \int_{\dOm} \frac{g_\e^\perp}{|g_\e|} \cdot \partial_\tau \bigg(\frac{g_\e}{|g_\e|}\bigg) \, d\h^1=0.$$ 
\end{thm}

 We believe that the power $\frac16-$ of $\eps$ in the above estimate is not optimal; moreover, the optimal power of $\eps$ is expected to be $\leq \frac12$ (see \eqref{delta} below). The proof of Theorem~\ref{thm:main} is done in several steps. In Section~\ref{s:apriori}, we obtain a preliminary estimate of the uniform convergence of $\abs{u_\e}$ to $1$, but at a much slower rate than the one claimed in Theorem~\ref{thm:main}. Thanks to this preliminary estimate, in Section~\ref{s:proof}, we  will be able to use more efficiently the Ginzburg-Landau system \eqref{eq:euler} to deduce an improved rate for the convergence of $\abs{u_\e}$ to $1$, first in the $L^2$-norm and then in the $L^\infty$-norm.

\medskip

Let us discuss the optimality of the assumption \eqref{regime} in Theorem \ref{thm:main}. First,
the assumption on $M_\eps$ in \eqref{regime} is optimal: if the constant $\kappa$ is not small enough, then solutions $u_\e$ of \eqref{eq:euler} may vanish inside $\Omega$. Moreover, the threshold value of $\kappa$ at which this happens can be arbitrarily small depending on the Lipschitz regularity of the domain:

\begin{pro}\label{p:optimal}
For any $\theta_0\in(0, \pi)$  and any $\eta>0$ there exists a cone shape domain $\Omega$ of opening angle $\theta_0$, an exponent $\alpha\in(0,1)$ 
 and a solution $u_\eps$ of the
Ginzburg-Landau system \eqref{eq:euler} such that for small $\e>0$, $u_\eps(P_\e)=0$ for a degree-one vortex point $P_\e\in \Omega$  and
\eqref{regime} holds true for $\kappa=\frac{\theta_0}{2}+\eta$.
\end{pro}

Second,
the assumption on $N_\e$ in \eqref{regime} is near-optimal in the following sense: if $N_\eps\lesssim \frac 1\e$, then a solution $u_\eps$ of \eqref{eq:euler} may have zeros at the boundary of any Lipschitz domain $\Omega$. 

\begin{pro}\label{p:optimal2}
For any Lipschitz domain $\Omega$, there exists a solution $u_\eps$ of the Ginzburg-Landau system \eqref{eq:euler}  such that for small $\e>0$, $u_\e(x_0)=0$ for some $x_0\in\dOm$, while 
$M_\e\lesssim 1$ and $N_\e \lesssim \frac1\e$.
\end{pro}

The proofs of Propositions~\ref{p:optimal} and \ref{p:optimal2} can be found in Section~\ref{sec:optimal}. The case $N_\eps\ll \frac1\eps$ (i.e., $\alpha=1$ in the regime \eqref{regime}) remains open; in that case, we conjecture that  our global estimate in Theorem \ref{thm:main} should still hold true, at least in smooth domains. 

\begin{rem}
Our proof adapts with minor modifications to critical points of the energy 
\begin{equation}
\label{enF}
E_\eps(u_\eps; \Omega):=\int_\Omega \frac 12\abs{\nabla u_\e}^2 + \frac{1}{4\e^2}F(\abs{u}^2)\: dx,
\end{equation}
where $F\in C^2([0,\infty))$ satisfies $F\geq 0$, $F(1)=0$ and  $(s-1)F'(s)\geq c(1-s)^2$ for all $s\geq 0$  and some constant $c>0$. The typical example is $F(s)=(1-s)^2$.
\end{rem}

\subsection{Related works}

There is a huge literature on the analysis of solutions $u_\eps$ of the Ginzburg-Landau system \eqref{eq:euler}. Let us only mention some of them (and apologize for omitting many other important ones). 

In the seminal paper \cite{BethuelBrezisHelein:1993a}, Bethuel, Brezis and H\'elein studied the system \eqref{eq:euler} on a smooth simply connected domain $\Om$ for minimizers $u_\eps$ of the associated energy functional, with a fixed smooth boundary data $g_\eps:=g$ such that $|g|=1$ on $\dOm$ and $g$ is of zero winding number (so $N_\eps, M_\eps$ are of order $1$); they proved that $\big|{|u_\eps|}-1\big|$ behaves as $\eps^2$ {\bf globally} in $\Om$ and this rate is optimal. They also studied the case of non-fixed smooth boundary data  $g_\eps:\dOm\to \R^2$ that is of zero winding number and has uniformly bounded energy $N_\eps\lesssim 1$; then for minimizers $u_\eps$, they deduced that $M_\eps\lesssim 1$ and $\big|{|u_\eps|}-1\big|$ behaves as $\eps^2$ {\bf locally} in  $\Om$.  These results also hold for non-minimizing solutions if $u_\e\to u_0$ 
strongly in $H^1$ for some limit $u_0$, see \cite[Remark A.1]{BethuelBrezisHelein:1994a}.

In \cite{BOS05}, Bethuel, Orlandi and Smets considered solutions of \eqref{eq:euler} that need not be minimizing, without imposing any bounds on $M_\e$ or $N_\e$. They proved local estimates on $\big|{|u_\eps|}-1\big|$, away from the boundary and from a vorticity set.
In our setting, their results imply that $\big|{|u_\eps|}-1\big|$ is of order at most $\e^{2(1-\beta)}M_\e$ in the region $\lbrace x\in\Omega\colon \dist(x,\partial\Omega)\geq\e^\beta\rbrace$, for any $\beta\in (0,1)$, but do not provide a good uniform estimate up to the boundary.

In the present work we focus on obtaining, for solutions of \eqref{eq:euler} that need not be minimizing, precise uniform estimates on $\big|{|u_\eps|}-1\big|$  which hold:
\begin{itemize}
\item up to the boundary $\partial\Om$ of a general Lipschitz domain, 
\item and in a regime that goes beyond the restrictive uniform bound $N_\e\lesssim 1$.
\end{itemize}

Estimates up to the boundary of a rectangle were obtained in \cite[Appendix]{Cote-Ignat-Miot}  in the regime $M_\e,N_\e\ll |\log\e|$. There it was proved that $\big|{|u_\eps|}-1\big|$ is of order at most $(\frac{1+N_\e+M_\e}{|\log\e|})^{\frac 16-}$ globally in $\Om$. In Section~\ref{s:apriori} we will follow the same approach in a general Lipschitz domain and under the less restrictive regime \eqref{regime}, as a first step towards the stronger estimate of Theorem~\ref{thm:main}.

\subsection{Motivation}

The energy functional $E_\e$ is a simplified version of a model describing superconducting materials. We simply mention here that  $\big|{|u_\eps|}-1\big|$ measures how close the system is to a superconducting state, and refer the interested reader to the monographs \cite{BethuelBrezisHelein:1994a,sandierserfaty}.

Another motivation comes from several studies of the pattern formation in thin ferromagnetic films \cite{IgnatOtto:2011a,Cote-Ignat-Miot,IK_bdv},
where one wishes to approximate
 $u_\eps$ by $\mathbb{S}^1$-valued maps away from the vortices. In a vortexless region $\Omega$ (assume e.g. $E_\eps(u_\eps; \Omega)\ll |\log \eps|$), the idea introduced in  \cite{IgnatOtto:2011a} consists in finding a (squared, spherical etc.) grid ${\cal R}_\eps$, each cell of the grid having the size $\sim \eps^\beta$ with $\beta\in (0,1)$ (i.e., much larger than the length-scale of a vortex) such that the energy $E_\eps(u_\eps; {\cal R}_\eps)$ on the $1$-dimensional grid ${\cal R}_\eps$ is of order $E_\eps(u_\eps; \Omega)/\eps^\beta$. Then Theorem \ref{thm:main} implies that $\big|{|u_\eps|}-1\big|$ behaves as a positive power of $\eps$ in $\Omega$, and $v_\eps=u_\e/|u_\e|$ is a \enquote{good} approximation of $u_\eps$ (in terms of the $L^2$ norm, their global  Jacobian etc., see \cite{IK_bdv}).
In that context, we give the following consequence of Theorem \ref{thm:main} for a cell of the grid leading to a key estimate needed in \cite{IK_bdv} (only a weaker version of this key estimate was needed 
in \cite{IgnatOtto:2011a,Cote-Ignat-Miot}):
\begin{cor}
Let ${\cal C}\subset \R^2$ be a Lipschitz bounded domain. Let $\eps>0$, $\beta\in (0,1)$ and 
${\cal C}_\eps:=\eps^\beta {\cal C}$ be a cell of size $\eps^\beta$. Assume that $u_\eps$ is a solution of \eqref{eq:euler} in ${\cal C}_\eps$ with
$$\int_{\partial {\cal C}_\eps} \frac12|\partial_\tau g_\eps|^2 +\frac1{4\eps^2} (1-|g_\eps|^2)^2\, d\h^1\ll \frac{|\log \eps|}{\eps^\beta}
$$
and 
$$\int_{{\cal C}_\eps} \frac12|\nabla u_\eps|^2 + \frac1{4\eps^2}(1-|u_\eps|^2)^2\, dx\ll |\log \eps|,$$
then $$\big|{|u_\eps|}-1\big|\leq C \eps^{\frac{1-\beta}6-} \quad \textrm{ in }\quad {\cal C}_\eps,$$ for some $C>0$ depending on  the Lipschitz regularity of ${\cal C}$. In particular, $g_\eps$ has zero winding number on ${\cal C}_\e$.
\end{cor}

\begin{proof}{} 
Denoting the rescaled map $\tilde u_{\tilde \eps}(\tilde x):=u_\eps(\eps^\beta \tilde x)$ for $\tilde x\in {\cal C}$ with $\tilde \eps:=\eps^{1-\beta}$, then $\tilde u_{\tilde \eps}$ satisfies the system \eqref{eq:euler} with the parameter $\tilde \eps$ instead of $\eps$ and the boundary energy, resp. interior energy of $\tilde u_{\tilde \eps}$ on $\partial {\cal C}$, resp. in ${\cal C}$ are estimated by $N_{\tilde \eps}, M_{\tilde \eps}\ll |\log \tilde \eps|$. By Theorem \ref{thm:main}, the conclusion follows.
\qed
\end{proof}

\medskip

As already hinted at, the regime \eqref{regime} is motivated by the study of boundary vortices (see e.g. \cite{Kurzke:2006b,IK_bdv}). The typical example is given by the formation of a dipole of two boundary vortices (in the absence of interior vortices).

\begin{exa}
\label{ex1}
Let $\Omega\subset \R^2$ be a bounded Lipschitz domain containing the upper half unit ball, more precisely, 
$$\Omega\cap B(0,1)=\{x=(x_1, x_2)\in B(0,1)\, :\, x_2>0\},$$
where $B(0,1)$ is the unit ball centered at the origin. Let $\eta=\eta(\e)\in (0,1)$ be a parameter. Consider the boundary data $g_\eps:\dOm\to \mathbb{S}^1$ such that $g_\eps(x)=e^{i\phi_\eps}$ with
$$\phi_\e(x)=\begin{cases}
0\quad & \textrm{ if } \,  x\in \dOm\setminus B(0, \eta),\\
\pi (1-\frac{|x_1|}\eta)\quad & \textrm{ if } \,  x=(x_1, x_2)\in \dOm\cap B(0, \eta).
\end{cases}$$
(This is the prototype of a dipole of two boundary vortices corresponding of two consecutive transitions between opposite directions $\tau$ and $-\tau$ at the boundary at a distance $\eta$).
We extend $\phi_\e$ to the entire domain $\Om$ by setting $\phi_\e=0$ in $\Om\setminus B(0, \eta)$
and $\phi_\e(x)=\pi (1-\frac{|x|}\eta)$ if $x\in \Om\cap B(0, \eta)$. Then we compute that
$$N_\eps=\int_{\partial \Omega} \frac12|\partial_\tau g_\eps|^2\, d\h^1\lesssim \frac1\eta, \quad E_\e(e^{i\phi_\e}; \Om)\lesssim 1.$$
Therefore, if $u_\eps$ is a minimizer of $E_\eps(\cdot; \Om)$ under the Dirichlet boundary condition $u_\eps=g_\eps$ on $\dOm$, we have that $E_\e(u_\e; \Om)\leq E_\e(e^{i\phi_\e}; \Om)$ so that \eqref{regime} holds provided that $\frac 1\eta\ll \frac1{\e^\alpha}$. In this case, Theorem \ref{thm:main} implies that $|u_\e|$ remains close to $1$ as a positive power of $\e$, in particular, no interior vortices appear in $\Om$.
\end{exa}

\medskip

\subsection*{Notations}

In the sequel we will use the symbol $\lesssim$ to denote an inequality up to a multiplicative constant that depends only on 
the Lipschitz regularity of $\Om$, that is, on $(\rho_0,\theta_0)\in (0,\infty)\times (0,\pi)$  such that for all $x\in\partial\Omega$ there is a cone of vertex $x$,  radius
$\rho_0$ and opening angle $\theta_0$ which is included in $\overline\Omega$, and the opposite cone is included in $\R^2\setminus\Omega$ (this is the uniform cone property, see e.g. \cite[Theorem~1.2.2.2]{grisvard}). We also note that, thanks to the uniform cone property, the rectangle 
\begin{equation*}
R=(-\frac {\rho_0}2 \sin\frac{\theta_0}{2},\frac {\rho_0}2 \sin\frac{\theta_0}{2} ) \times (-\rho_0\cos\frac{\theta_0}{2},\rho_0\cos\frac{\theta_0}{2}),
\end{equation*}
 has the following property: for all $x\in\Omega$, there exists an angle $\gamma=\gamma(x)\in\R$ such that for all $t\in (0,1]$, the set
\begin{equation}\label{eq:Rt}
\mathcal R_t(x)=(x+t e^{i\gamma}  R)\cap\Omega\text{ is bi-Lipschitz homeomorphic to }tB,
\end{equation}
where $B$ is the unit ball, and the Lipschitz constants of the homeomorphism and its inverse are bounded by a constant depending only on $(\rho_0,\theta_0)$. See Figure~\ref{fig:cone} below.

\begin{figure}[h]
\caption{Cone property and rectangle $\mathcal R_1(x)$ at a boundary point $x\in\partial\Omega$}\label{fig:cone}
\begin{center}
\begin{tikzpicture}[scale=.4]

\draw[gray] (-3,-4) -- (3,4);
\draw[gray] (-3,4) -- (3,-4);
\draw[gray,thick,->] (.5,0) -- (3.5,3.6);
\draw[gray, thick] (3.2,2.3) node {$\rho_0$};
\draw[gray] (3,4) arc (53:127:5);
\draw[thick,gray,<->] (-3.2,4.5) arc (127:53:5.2);
\draw[thick, gray] (0,6.1) node {$\theta_0$};
\draw[gray] (-3,-4) arc (-127:-53:5);
\draw[thick] (1.5,4) -- (-1.5,4) -- (-1.5,-4) -- (1.5,-4) -- (1.5,4);
\draw[thick] (0,2.8) node { $\mathcal R_1(x)$};
\draw[very thick] (-3.5,1) -- (-2.3,.3) -- (-1.3,.8) -- (0,0) -- (1,1) -- (2.5,.5) -- (3.5,1.5);
\draw[very thick] (3.9,.9) node { $\partial\Omega$};
\draw (0,0) node {$\bullet$} node [below] {$x$};

\end{tikzpicture}
\end{center}
\end{figure}
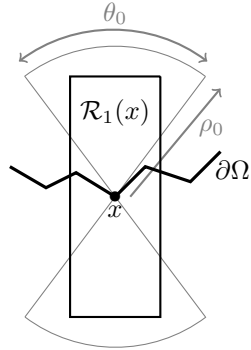

We recall that for $a\in\R$ we write $a+$ (resp. $a-$) to denote any real number strictly greater (resp. smaller) than $a$ but that can be chosen arbitrarily close to $a$. In inequalities involving such exponents, the constant will also depend on the distance of that number to $a$.

We write $B(x,r)$ for the ball centered at $x$ of radius $r$.

\section{A-priori global uniform estimate of $|u_\eps|$ in $\Om$}\label{s:apriori}
The aim of this section is to prove the following weaker estimate of $\big|{|u_\eps|}-1\big|$ in $\Om$:

\begin{thm}
\label{thm:apriori} 
Let $\Omega\subset \R^2$ be a Lipschitz bounded domain. 
If $u_\e$ satisfies \eqref{eq:euler} and \eqref{regime}, then we have
$$\sup_{\Omega}\big|{ |u_\eps|}-1\big|
  \lesssim \left( \frac{1+ M_\eps}{|\log \eps|} \right)^{\frac16-}. $$ 
 In particular, if  $\kappa$ is small enough in \eqref{regime} then $|u_\eps| \geq \frac12$ in $\Omega$  as $\eps\to 0$.
 \end{thm}

\nd Theorem \ref{thm:apriori} is an improvement of \cite[Theorem 6 in Appendix]{Cote-Ignat-Miot}, 
where the boundary data satisfies the additional condition $|g_\eps|\leq 1$, $\Omega$ is a square and $N_\eps\ll |\log \eps|$. We will follow the strategy in \cite{Cote-Ignat-Miot}, generalizing  to Lipschitz domains and general boundary data $g_\eps:\partial \Omega \to \R^2$ with $N_\eps$ satisfying the  wider regime \eqref{regime}. The proof of Theorem \ref{thm:apriori} is divided 
into three parts:

\bigskip

\nd {\bf Part 1 of the proof of Theorem \ref{thm:apriori} }. We prove the following upper bound of $|u_\eps|$ in $\Omega$:
\be
\label{upper}
\|u_\eps\|_{L^\infty(\Omega)}-1\lesssim \sqrt{\eps N_\eps}. 
\ee
For that, we start by denoting $\zeta=(1-|g_\eps|)^2$ on $\dOm$. The Cauchy-Schwarz inequality yields:
\footnote{For the more general energy \eqref{enF}, only the estimate $F(s)\gtrsim (1-s)^2$ is needed, which is a consequence of $(s-1)F'(s)\gtrsim (1-s)^2$ and $F(1)=0$.}
\begin{align*}
\frac12|\partial_\tau g_\eps|^2 +\frac1{4\eps^2} (1-|g_\eps|^2)^2&\geq \frac1{8\eps^2}\zeta+\bigg(\frac1{8\eps^2}\zeta+\frac12\big|\partial_\tau|g_\eps|\big|^2\bigg)\geq \frac1{8\eps^2}\zeta+\frac1{4\eps}|\partial_\tau\zeta| \quad \textrm{on } \partial \Om.
\end{align*}
Using the embedding $W^{1,1}(\partial \Om)\subset L^\infty(\partial \Om)$, as $\h^1(\partial \Om)\geq \eps$, we deduce by \eqref{bdry}:
$$
N_\eps=\int_{\partial \Omega} \frac12|\partial_\tau g_\eps|^2 +\frac1{4\eps^2} (1-|g_\eps|^2)^2\, d\h^1\gtrsim \frac{1}{\eps} \|\zeta\|_{L^\infty(\dOm)}, \quad \textrm{as } \quad \eps\to 0,
 $$
so that 
\be
\label{delta}
\delta_\eps:=\big\|{ |g_\eps|}-1\big\|_{L^\infty(\dOm)}\lesssim \sqrt{\eps N_\eps}. 
\ee
Let $\tilde \rho_\eps=1-|u_\eps|^2$ in $\Om$. Then \eqref{eq:euler} implies that 
$$-\Delta \tilde \rho_\eps+\frac{2}{\eps^2}|u_\eps|^2\tilde \rho_\eps=2|\nabla u_\eps|^2\geq 0 \quad \textrm{ in }\quad \Omega$$ 
and $\tilde \rho_\eps=1-|g_\eps|^2\geq 1-(1+\delta_\eps)^2$ on $\partial \Omega$.
Thus, the maximum  principle\footnote{This argument adapts to critical points of the general  energy \eqref{enF}, provided $F'(s)\geq 0$ for $s\geq 1$, see e.g.  \cite[Lemma~8.3]{lamy14}.}
 implies that $\tilde \rho_\eps\geq 1-(1+\delta_\eps)^2$ in $\Om$, i.e., $|u_\eps|\leq 1+\delta_\eps$ in $\Omega$ yielding \eqref{upper} by \eqref{delta}.

\bigskip

\nd {\bf Part 2 of the proof of Theorem \ref{thm:apriori} }. We estimate a H\"older seminorm for $u_\eps$.

\begin{lem}\label{lemma:maximum} Let $\Omega\subset \R^2$ be a Lipschitz bounded domain. If $u_\eps$ satisfies \eqref{eq:euler} and \eqref{regime}, then \footnote{For the general 
energy \eqref{enF} this argument only uses the fact that $F$ is $C^1$ and the validity of a uniform upper bound $\norm{u_\e}_\infty\lesssim 1$, implied e.g. by \eqref{upper} which is valid as soon as $F'(s)\geq 0$ for $s\geq 1$.}
 \[
|u_\eps(x)-u_\eps(y)|\leq C \bigg(\frac{|x-y|}{\eps}\bigg)^{\frac12-} \qquad \forall x, y\in \Omega,
\] 
where $C>0$ depends only on the Lipschitz regularity of $\Om$.
\end{lem}

\begin{rem}
In general, we don't have that $\|\nabla u_\eps\|_{L^\infty(\Omega)}\leq \frac{C}\eps$ because this estimate can be violated by the boundary condition $g_\eps$ on $\dOm$. But since $g_\eps$ belongs to $H^1(\dOm)$ that embeds into the H\"older space $C^{0, \frac12}(\dOm)$, we 
can
deduce an appropriate estimate of a H\"older seminorm for $u_\eps$ in $\Omega$.
\end{rem}

\begin{proof}{ of Lemma \ref{lemma:maximum}} 
Consider the rescaled map $\hat u(\hat x)=u_\e(\e \hat x)$ defined for $\hat x\in\Omega_\e=\e^{-1}\Omega$.  (The map $\hat u$ depends on $\eps$, we omit this dependence to simplify notation.) This map solves
\begin{equation*}
\left\lbrace
\begin{aligned}
-\Delta \hat u & =(1-\abs{\hat u}^2)\hat u &\text{in }\Omega_\e,\\
\hat u & =\hat g &\text{on }\partial\Omega_\e,
\end{aligned}
\right.
\end{equation*}
where $\hat g(\hat x)=g_\e(\e\hat x)$ for $\hat x\in\partial\Omega_\e$.
We fix $x_0\in\Omega_\e$ and consider the Lipschitz domain 
\begin{equation*}
\mathcal R =\mathcal R(x_0) = \frac 1\e \left( (\e x_0 +\e e^{i\gamma(\e x_0)} R)\cap\Omega\right),
\end{equation*}
which is bi-Lipschitz homeomorphic to the unit ball $B$, with Lipschitz bounds uniform in $\e$ and $x_0$ and depending 
only on the Lipschitz regularity of $\Omega$,  thanks to the definition of $R$, see \eqref{eq:Rt}.
 Since $|\hat g|\leq 1+\delta_\e\leq 2$ on $\partial \Om_\eps$ (by \eqref{delta}) and $\abs{\hat u}\leq 1+\delta_\e\leq 2$ in $\Om_\eps$ (by \eqref{upper}) as $\eps\to 0$, elliptic estimates in Lipschitz domains (see e.g. \cite{JK81,savare98}, and \cite[Section VI]{jonssonwallin} for the theory of traces) yield
\begin{equation*}
\norm{\hat u}_{H^{\frac 32 -}(\mathcal R)}\lesssim 1+\norm{\hat g}_{\dot{H}^{1}(\partial\Omega_\e)}.
\end{equation*}
 The constant depends only on the Lipschitz regularity of the domain $\mathcal R$ (see e.g. the proof of Theorem~2 in \cite{savare98}), and is therefore bounded independently of $x_0\in\Omega_\e$ and $\e\in (0,1]$.
 By Sobolev embedding we deduce that
\begin{equation*}
\|\hat u\|_{C^{0,\frac 12 -}(\mathcal R)}\lesssim 1 + \norm{\hat g}_{\dot{H}^1(\partial\Omega_\e)} \lesssim 1+(\e N_\e)^{\frac12}.
\end{equation*}
The constant in the Sobolev imbedding depends only on the Lipschitz regularity of $\Omega$, since the imbedding inequalities $\norm{v}_{L^{4-}(B)}\lesssim \norm{v}_{H^{\frac 12 -}(B)}$ 
and $\|v\|_{C^{0,\frac 12-}(B)}\lesssim \norm{v}_{W^{1,4-}(B)}$ are valid on the unit ball $B\subset \R^2$ and behave well under composition by a bi-Lipschitz homeomorphism.
Since any two points $x,y\in\Omega_\e$ which are close enough are contained in a domain $\mathcal R(x_0)$ for some $x_0\in\Omega_\e$,  recalling once more that $\abs{\hat u}\leq 2$ in $\Om_\eps$ (by \eqref{upper}) we infer 
\begin{equation*}
\norm{\hat u}_{C^{0,\frac 12 -}(\Omega_\e)}\lesssim 1+(\e N_\e)^\frac12 \lesssim 1 \quad \textrm{as} \quad \eps\to 0.
\end{equation*}
The last inequality is due to our assumption \eqref{regime}.
The conclusion follows by scaling back to $u_\e(x)=\hat u(\e^{-1}x)$.
\qed
\end{proof}

\bigskip

\nd {\bf Part 3 of the proof of Theorem \ref{thm:apriori}}. 
We start by estimating the normal derivative of $u_\eps$ at the boundary $\dOm$: 
 
\begin{lem}
\label{lem_normal}
Let $\Omega\subset \R^2$ be a Lipschitz bounded domain. 
If $u_\eps$ satisfies \eqref{eq:euler}, then we have \footnote{In the context of the general 
energy \eqref{enF}, we need only the assumption that $F\in C^1$.}
$$\int_{\partial \Om} |\partial_\nu u_\eps|^2\, d\h^1 \lesssim M_\eps+N_\eps.$$
\end{lem}

\proof{ of Lemma \ref{lem_normal}} We use the Pohozaev identity for $u_\eps$ in the spirit of \cite[Proposition 3]{BethuelBrezisHelein:1993a}, the only difference is to adapt that result to the setting of Lipschitz domains $\Om$.  More precisely, we consider a map $V:\Omega\to \R^2$ that is $C^1$ in the closed domain $\bar \Omega$ and such that 
$V\cdot \nu\geq a>0$ on $\partial \Omega$ for some $a>0$ depending only on the Lipschitz regularity of $\Omega$
(see e.g. \cite[Lemma~1.5.1.9]{grisvard}).
Multiplying the equation \eqref{eq:euler} by $(V(x)\cdot \nabla) u_\eps$ and integrating by parts,
as $V\in C^1(\bar \Om)$, we deduce by \eqref{bdry} and \eqref{int}:
\begin{align}
\big|\frac{1}{\eps^2}\int_{ {\Omega}} u_\eps(1-|u_\eps|^2) \cdot (V(x)\cdot \nabla) u_\eps\, dx \big|\nonumber&=\bigg|\frac{1}{4\eps^2}\int_{ {\Omega}} \di V(1-|u_\eps|^2)^2 \, dx \\
\label{unu}&\quad \quad -\frac{1}{4\eps^2}\int_{\partial {\Omega}} V(x)\cdot \nu (1-|g_\eps|^2)^2 \, d\h^1\bigg| 
\lesssim M_\eps+N_\eps,\\
 \int_{ {\Omega}} \Delta u_\eps \cdot (V(x)\cdot \nabla) u_\eps\, dx 
\label{doi}&=\int_{ \partial {\Omega}} \bigg( (\nu \cdot \nabla)u_\eps \cdot (V\cdot \nabla) u_\eps 
-\frac12 V\cdot \nu |\nabla u_\eps|^2 \bigg)\, d\h^1\\
\nonumber&\quad \quad +\int_\Om \bigg( \frac12 \di V |\nabla u_\eps|^2-\sum_{j=1,2} \partial_j u_\eps \cdot (\partial_j V\cdot \nabla)u_\eps\bigg)\, dx.
\end{align}
 For $x\in \partial {\Omega}$, we decompose
$V=s(x)\nu+t(x)\tau$ where $s,t\in L^\infty(\dOm)$, $s(x)=V\cdot \nu \geq a>0$ for a.e. $x\in \dOm$, and
$\nabla u_\eps=\nu\otimes \partial_\nu u_\eps+\tau \otimes \partial_\tau g_\eps$ on $\partial {\Omega}$. By \eqref{eq:euler}, \eqref{bdry}, \eqref{int}, \eqref{unu} and \eqref{doi}, as $V\in C^1(\bar \Om)$, we conclude  by Young's inequality:
$$M_\eps+N_\eps \gtrsim \int_{ \partial {\Omega}} \bigg(\frac{s(x)}2 \big|\partial_\nu u_\eps|^2
-\frac{s(x)}2 \big|\partial_\tau g_\eps\big|^2+t(x) \partial_\nu u_\eps \cdot \partial_\tau g_\eps \bigg)\, d\h^1 
\gtrsim \int_{ \partial {\Omega}} \big|\partial_\nu u_\eps|^2\, d\h^1-N_\eps.$$
\qed

\bigskip

We use Lemma \ref{lem_normal} to prove the following estimate of the potential energy in small balls (of radius $\ll \eps^\alpha$). To simplify notation,  we denote the energy density by
$$ \den(u_\eps):=\frac12|\nabla u_\eps|^2 + \frac1{4\eps^2}(1-|u_\eps|^2)^2, \quad u_\eps:\Omega\to \R^2.$$
(In the context of the energy \eqref{enF}, only the assumption $F\in C^1$ is needed for the following estimate).

\begin{lem}
\label{lem:localized}
Let $\Omega\subset \R^2$ be a Lipschitz bounded domain and $u_\eps$ be a solution of \eqref{eq:euler} in the regime \eqref{regime}. 
Fix $1>\bun>\bd>\alpha>0$. There exists $C\geq 1$ such that for every $x_0\in \Omega$, we can find $r_0=r_0(x_0)\in (\eps^{\bun}, \eps^{\bd})$
such that
\be
\label{ineq1}
\int_{\partial \big(B(x_0,r_0)\cap \Omega\big)} \den(u_\eps)\, d\h^1\leq \frac{C(1+ M_\eps)}{r_0 |\log \eps|}\ee
for every $\eps\leq \eps_0$ with $\eps_0=\eps_0(\bd, \alpha)>0$. 
Moreover, we have that
\be \label{ineq2}
\frac{1}{\eps^2} \int_{B(x_0,r_0)\cap \Omega} (1-|u_\eps|^2)^2\, dx \leq \frac{\tilde{C}(1+ M_\eps)}{|\log \eps|}\ee
for some $\tilde{C}\geq1$.
\end{lem}    

\begin{proof}{ of Lemma \ref{lem:localized}} We distinguish two steps:

\medskip

\nd {\bf Step 1}. {\it Proof of \eqref{ineq1}.} Assume by contradiction that for every $C\geq 1$ there exists $x\in \Omega$ such that
for every $r\in (\eps^{\bun}, \eps^{\bd})$
we have $$\int_{\partial \big(B(x, r)\cap \Omega\big)} \den(u_\eps)\, d\h^1\geq \frac{C (1+M_\eps)}{r |\log \eps|}.$$
Since $N_\eps \eps^\alpha\ll 1$, by \eqref{bdry} and Lemma \ref{lem_normal}, there exists $c_1>0$ such that $$\int_{\partial \Omega} \den(u_\eps)\, d\h^1\leq c_1( M_\eps+N_\eps)\leq  \frac{1+ M_\eps}{2 \eps^{\bd} |\log \eps|}\leq \frac{C(1+M_\eps)}{2r |\log \eps|}, \quad \forall r\in (\eps^{\bun}, \eps^{\bd})$$
for every $\eps\leq \eps_0$ (with $\eps_0>0$ depending on $\bd$ and $\alpha$).
Therefore, we deduce that
$$\int_{\partial B(x, r)\cap \Omega} \den(u_\eps)\, d\h^1\geq \frac{C(1+ M_\eps)}{2r |\log \eps|}.$$
Integrating in $r\in (\eps^{\bun}, \eps^{\bd})$, we obtain by \eqref{int}:
$$M_\eps=\int_{\Omega} \den(u_\eps)\, dx\geq \int_{B(x, \eps^{\bd})\cap \Omega} \den(u_\eps)\, dx\geq \int_{\eps^{\bun}}^{\eps^{\bd}} 
dr \int_{\partial B(x, r)\cap \Omega} \den(u_\eps)\, d\h^1\geq \frac{C(\bun-\bd)(1+ M_\eps)}{2}$$
which is a contradiction with the fact that $C$ can be arbitrary large.

\bigskip

\nd {\bf Step 2.} {\it Proof of \eqref{ineq2}.} 
Let $\nu$ be the outer unit normal vector at  the boundary of the domain$$ \cd:=B(x_0,r_0)\cap \Omega.$$
As in the proof of Lemma \ref{lem_normal}, we use the Pohozaev identity for the solution $u_\eps$ of
\eqref{eq:euler}. Indeed, multiplying the equation by $(x-x_0)\cdot \nabla u_\eps$ and integrating by parts, 
we deduce:
\begin{align*}
\int_{ \cd} -\Delta u_\eps \cdot \bigg((x-x_0)\cdot \nabla u_\eps\bigg)\, dx&=\int_{ \partial \cd} \bigg( \frac 1 2 (x-x_0)\cdot \nu |\nabla u_\eps|^2
-\partial_\nu u_\eps\cdot \big((x-x_0)\cdot\nabla\big) u_\eps
\bigg)\, d\h^1,\\
\frac{1}{\eps^2}\int_{ \cd} u_\eps(1-|u_\eps|^2) \cdot \bigg((x-x_0)\cdot \nabla u_\eps\bigg)\, dx&=\frac{1}{2\eps^2}\int_{ \cd} (1-|u_\eps|^2)^2 \, dx\\&\quad \quad -\frac{1}{4\eps^2}\int_{\partial \cd} (x-x_0)\cdot \nu (1-|u_\eps|^2)^2 \, d\h^1.
\end{align*}
Since $|x-x_0|\leq r_0$ on $\partial \cd$, by \eqref{ineq1},
we deduce that \eqref{ineq2} holds true.
\qed
\end{proof}

\bigskip

The conclusion of Theorem \ref{thm:apriori} comes from the following result:

\begin{lem}\label{lemma:lower-bound} 
Let $\Omega\subset \R^2$ be a Lipschitz bounded domain. If $u_\eps$ satisfies \eqref{eq:euler} and \eqref{regime}, then we have
\footnote{For the general energy \eqref{enF} we only need here $(s-1)^2\lesssim F(s)$.}
$$\||u_\eps|^2-1\|_{L^\infty(\Omega)}\lesssim \left( \frac{1+ M_\eps}{|\log \eps|} \right)^{\frac16-}.$$
\end{lem}

\begin{proof}{}
Let $x_0\in \Omega$ 
and set $1>A\geq 0$ such that
$$4C\: A^{\frac12-}=\frac{\big|1-|u_\eps(x_0)|^2\big|}{2},$$ 
where $C\geq 1$ is given by Lemma \ref{lemma:maximum}.  By Lemma \ref{lemma:maximum}, we obtain for any $y\in B(x_0,A\eps)\cap \Omega$
$$\big|1-|u_\eps(y)|^2\big|\geq \big|1-|u_\eps(x_0)|^2\big|-4C\: A^{\frac12-}=\frac{\big|1-|u_\eps(x_0)|^2\big|}{2}$$
as $|u_\eps(y)|+|u_\eps(x_0)|\leq 4$. Hence, for small $\eps$,
\begin{equation}\label{ineq:reuse}
\begin{aligned}\int_{B(x_0, A\eps)\cap \Omega}(1-|u_\eps(y)|^2)^2\,dy & \geq C(\Omega)A^2\eps^2 (1-|u_\eps(x_0)|^2)^2\\
&\geq \tilde C(\Omega) \eps^2 (1-|u_\eps(x_0)|^2)^{6+},
\end{aligned}
\end{equation}
where $C(\Omega),\tilde C(\Omega)>0$.
We have that $B(x_0, A\eps)\subset B(x_0, r_0)$ for $\eps\leq \eps_0$ with $\eps_0$ depending only on $\alpha_1$ in Lemma \ref{lem:localized}. Thus, by  \eqref{ineq2}, we obtain
$$
(1-|u_\eps(x_0)|^2)^{6+}\leq  \hat{C}\frac{1+ M_\eps}{|\log \eps|} $$
and the conclusion follows.
\qed
\end{proof}

\section{Proof of Theorem \ref{thm:main}}\label{s:proof}

The main idea is to improve the convergence of $|u_\eps|$ to $1$ locally in $L^2$-norm; this consists in improving the local estimate of the potential energy \eqref{ineq2} to a positive power of $\eps$ and then the argument in Lemma \ref{lemma:lower-bound} yields the conclusion (i.e., the desired estimate in $L^\infty$-norm of  $|u_\eps|-1$ in our main result).

Let $x_0\in\Om$ and $\e>0$.  By Fubini's theorem we may choose $t\in [1/2,1]$ such that the domain 
\be
\label{dr}
\mathcal R =\mathcal R_t(x_0)
\ee
defined in \eqref{eq:Rt} satisfies
\begin{equation}\label{choi}
\int_{\partial \mathcal R\cap\Om} \frac12|\nabla u_\eps|^2 +\frac1{4\eps^2} (1-|u_\eps|^2)^2\, d\h^1\lesssim  M_\eps.
\end{equation}
Recall that $\mathcal R$ is bi-Lipschitz homeomorphic to the unit ball $B$, in particular it is simply connected. 
Moreover by Theorem~\ref{thm:apriori} if $\kappa$ is small enough then $u_\e$ does not vanish. So we may write
\begin{equation*}
u_\eps=\rho_\eps e^{i\f_\eps} \quad \text{in } \mathcal R,
\end{equation*}
with $\rho_\eps, \f_\eps\in H^1(\mathcal R)$ (moreover, $\rho_\eps^2$ and $\f_\eps$ are smooth in $\mathcal R$ as $u_\eps$ is smooth by standard elliptic regularity). The system \eqref{eq:euler} writes in terms of $\rho_\eps$ and $\f_\eps$:
\be
\label{PDE_rho_phi}
\begin{cases}
&-\Delta \rho_\eps+\rho_\eps|\nabla \f_\eps|^2=\frac1{\eps^2} \rho_\eps (1-\rho_\eps^2)\\
&\di (\rho_\eps^2 \nabla \f_\eps)=0
\end{cases}
\quad \textrm{in } \mathcal R.
\ee

\nd {\bf Step 1.} {\it We prove the following estimate
\footnote{For the general energy \eqref{enF}, no modification is required for this step since the equation satisfied by $\varphi_\e$ stays the same.}
 of $\nabla \f_\eps$ in $L^q(\mathcal R)$, where $q=4-$: 
\be
\label{phi_w1q}
 \|\nabla \f_\eps\|_{L^q(\mathcal R)}\lesssim 1+{N_\eps}^{\frac12}+ {M_\eps}^{\frac12}.
 \ee
}
Indeed, by \eqref{bdry}, \eqref{delta}, Lemma \ref{lem_normal} and \eqref{choi}, we note that
\be
\label{auxi}
\begin{aligned}
\int_{\partial \Om\cap  \mathcal R} |\nabla \f_\eps|^2\, d\h^1 & \lesssim \int_{\partial \Om} |\nabla u_\eps|^2\, d\h^1\lesssim N_\eps+M_\eps \\
\text{ and } \, \int_{\Om\cap \partial \mathcal R} |\nabla \f_\eps|^2\, d\h^1 &\lesssim \int_{\partial \mathcal R \cap \Om} |\nabla u_\eps|^2\, d\h^1\lesssim M_\eps.
\end{aligned}
\ee
Therefore, by the Poincar\'e-Wirtinger inequality, up to adding a constant to $\varphi_\e$, we can assume that 
\be
\label{estim_lift}
\|\f_\eps\|_{H^1(\partial \mathcal R)}\lesssim 1+N_\eps^{\frac12}+ {M_\eps}^{\frac12}.
\ee
By the theory of traces in Lipschitz domains (see e.g. \cite[Section VI.2]{jonssonwallin}), for
  $s=1-$ there is a continuous extension operator from $H^s(\partial \mathcal R)$ to $H^{s+1/2}(\mathcal R)$, and its operator norm is bounded by a constant depending only on the Lipschitz regularity of $\mathcal R$, hence only on the Lipschitz regularity of $\Omega$. Thus there exists an extension $\Phi\in H^{\frac 32 -}(\mathcal R)$ of $\f_\eps\big|_{\partial \mathcal R}$ such that 
\begin{equation*}
\norm{\Phi}_{H^{\frac 32 -}(\mathcal R)}\lesssim 1+ N_\e^{\frac 12}+ {M_\eps}^{\frac12}.
\end{equation*} 
By Sobolev embedding $H^{\frac 12-}(\mathcal R)\subset L^{4-}(\mathcal R)$ we deduce the bound
\begin{equation}\label{estim_ext}
\norm{\nabla \Phi}_{L^q(\mathcal R)}\lesssim 1+ N_\e^{\frac 12}+ {M_\eps}^{\frac12}.
\end{equation}
The constant in the Sobolev embedding depends only on the Lipschitz regularity of $\Omega$ since $\mathcal R$ is bi-Lipschitz homeomorphic to the unit ball (with Lipschitz constants depending only on the Lipschitz regularity of $\Omega$).
Denoting
$$\psi:=\varphi_\e-\Phi\in H^1_0(\mathcal R),$$
by \eqref{PDE_rho_phi}, $\psi$ solves
\begin{equation*}
\Delta\psi=\nabla\cdot \left( (1-\rho_\e^2)\nabla\varphi_\e -\nabla\Phi\right)\quad\text{in }\quad \mathcal R,
\end{equation*}
so that elliptic estimates in Lipschitz domains \cite{JK81,savare98} yield
\begin{equation*}
\norm{\nabla\varphi_\e}_{L^q(\mathcal R)}\leq C(1 + \norm{(1-\rho_\e^2)\nabla\varphi_\e}_{L^q(\mathcal R)} + \norm{\nabla \Phi}_{L^q(\mathcal R)}).
\end{equation*}
By Theorem~\ref{thm:apriori}, $C\abs{1-\rho_\e^2}\leq \frac12$ in $\mathcal R$ for $\kappa>0$ small enough.  This implies
\begin{equation*}
\norm{\nabla\varphi_\e}_{L^q(\mathcal R)}\lesssim 1 + \norm{\nabla \Phi}_{L^q(\mathcal R)}.
\end{equation*}
The last term can be estimated by \eqref{estim_ext} and this proves \eqref{phi_w1q}.

\bigskip

\nd {\bf Step 2.} {\it An improved local estimate
 of the potential energy}. 
We will prove the following:

\begin{lem}
\label{lem:pot_en}
Let $\Omega\subset \R^2$ be a Lipschitz bounded domain. 
If $u_\eps$ satisfies \eqref{eq:euler} and \eqref{regime}, then 
$$\frac1{\eps^2} \int_{\mathcal R} (1-|u_\eps|^2)^2\, dx\lesssim \eps^{1-}(1+N_\eps+M_\eps)(1+M_{\eps})^{\frac12-},$$
for every point $x_0\in \Om$ with the associated domain $\mathcal R$ in \eqref{dr}.
\end{lem}

\begin{proof}{}
Multiplying \eqref{PDE_rho_phi} by $1-\rho_\eps^2$, as $\rho_\eps\geq 1/2$ in $\mathcal R$ (by Theorem \ref{thm:apriori}), integration by parts yields\footnote{For the general energy \eqref{enF}, this estimate holds thanks to the assumption $(s-1)F'(s)\gtrsim (s-1)^2$ for $s\geq 0$.} 
\begin{align*}
\frac{1}{2\e^2}\int_{\mathcal R} (1-\rho_\eps^2)^2 \, dx & \leq \frac{1}{\e^2}\int_{\mathcal R} \rho_\eps(1-\rho_\eps^2)^2\, dx\\
&=-\int_{\mathcal R} (1-\rho_\eps^2)\Delta\rho_\eps \, dx
+\int_{\mathcal R} \rho_\eps(1-\rho_\eps^2)|\nabla\varphi_\eps|^2\, dx \\
&= -\int_{{\partial \mathcal R}}  (1-\rho_\eps^2)\partial_\nu\rho_\eps \, d\h^1- 2\int_{\mathcal R}\rho_\eps\abs{\nabla\rho_\eps}^2 \, dx +\int_{\mathcal R} \rho_\eps(1-\rho_\eps^2)\abs{\nabla\varphi_\eps}^2\, dx\\
&\leq \norm{1-\rho_\eps^2}_{L^2({\partial \mathcal R})} \norm{\partial_\nu\rho_\eps }_{L^2({\partial \mathcal R})}
 + 2\norm{\nabla\varphi_\eps}_{L^q(\mathcal R)}^2\norm{1-\rho_\eps^2}_{L^{\frac{q}{q-2}}(\mathcal R)}\\
&\lesssim \eps (M_\eps+N_\eps) +\eps^{1-}(1+N_\eps+ M_\eps)M_{\eps}^{\frac12-}
\end{align*}
for $q=4-$, where we used 

$\bullet$ \eqref{bdry} and \eqref{choi} yielding $\norm{1-\rho_\eps^2}_{L^2({\partial \mathcal R})}\lesssim  \eps (M_\eps+N_\eps)^{\frac12}$;

$\bullet$ \eqref{auxi} yielding $ \norm{\partial_\nu\rho_\eps }_{L^2({\partial \mathcal R})}\lesssim (M_\eps+N_\eps)^{\frac12}$;

$\bullet$ \eqref{phi_w1q}
and the interpolation inequality for $\lambda=\frac{2(q-2)}q=1-$
$$\norm{1-\rho_\eps^2}_{L^{\frac{q}{q-2}}(\mathcal R)}\leq \norm{1-\rho_\eps^2}_{L^\infty(\mathcal R)}^{1-\lambda}
\norm{1-\rho_\eps^2}_{L^2(\mathcal R)}^\lambda\stackrel{\eqref{int},\eqref{upper}}{\lesssim} \e^\lambda M_\e^{\frac\lambda2}$$
yielding the last estimate.
\qed
\end{proof}

\medskip

\nd {\bf Step 3}. {\it Conclusion of Theorem \ref{thm:main}}. 
Applying the arguments in the proof of Lemma~\ref{lemma:lower-bound} in the domain $\mathcal R=\mathcal R_t(x_0)$ defined at \eqref{dr}, we find
\begin{equation*}
(\abs{u_\e(x_0)}^2-1)^{6+} \lesssim \frac{1}{\e^2}\int_{\mathcal R} (1-\rho_\eps^2)^2 \, dx\lesssim
\eps^{1-}(1+N_\eps+ M_\eps)(1+M_{\eps})^{\frac12-}.
\end{equation*}
The last inequality follows from the previous step. Since $x_0\in\Omega$ is arbitrary and the constant depends only on the Lipschitz regularity of $\Omega$, this proves Theorem~\ref{thm:main}. 
\qed

\medskip

\section{Optimality of the regime \eqref{regime}}
\label{sec:optimal}

In this section, we prove Propositions \ref{p:optimal} and \ref{p:optimal2}:
 
\begin{proof}{ of Proposition~\ref{p:optimal}}
Let $\Omega$ be a cone of opening angle $\theta_0$ and height $1$, see Figure \ref{fig1}. 
Consider the point $P_\e$ on the medial axis at distance $s_\eps$ from the corner, where
$$s_\e=\e^\mu \quad \textrm{ with } 0<\mu<1.$$
Set $\alpha=\frac{1+\mu}{2}\in (0,1)$. We also denote by $d_\eps$ the distance of $P_\eps$ to the boundary $\dOm$. For $\theta_1=\theta_0+\eta$ (where, possibly lowering $\eta$, we may assume $\eta<\theta_0$) consider the cone $K_1$ of opening $\theta_1$ and height $1$ centered at $P_\e$ and with the same medial axis.
The boundaries of the two cones intersect in two points at a distance $r_\e$ from $P_\e$. It follows 
that $\Omega\subset B(P_\e, {r_\e})\cup K_1$ (as $s_\eps< r_\eps$),
$$d_\e=s_\e\sin \frac{\theta_0}{2}\sim \eps^\mu \quad \textrm{and}\quad r_\e=s_\e \frac{\sin \frac{\theta_0}{2}}{\sin \frac\eta2}\sim \eps^\mu.$$

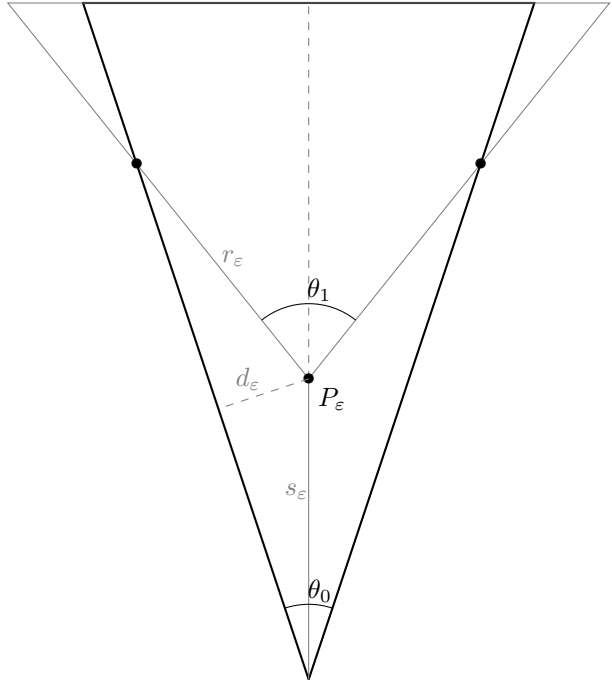
\begin{figure}[htbp]
\label{fig1}
\caption{The cones $\Omega$ and $K_1$ of opening angles $\theta_0$ and $\theta_1$ respectively.}

\begin{center}
\begin{tikzpicture}[scale=1]

\draw[thick] (0,0) -- (3,9) -- (-3,9) -- (0,0);
\draw[gray] (0,0) -- (0,4);
\draw[gray,dashed] (0,4) -- (0,9);
\draw[gray] (0,4) -- (4,9) -- (-4,9) -- (0,4);
\draw (0,4) node {$\bullet$} node[below right] {$P_\varepsilon$};
\draw (2.286,6.857) node {$\bullet$};
\draw (-2.286,6.857) node {$\bullet$};
\draw (0,0)++(71.565:1) arc (71.565:108.435:1);
\draw (0.15,1.18) node {$\theta_0$};
\draw (0,4)++(51.34:1) arc (51.34:128.66:1);
\draw (0.14,5.18) node {$\theta_1$};
\draw[gray,dashed] (0,4) -- (-1.2,3.6);
\draw[gray] (-.8,4) node {$d_\varepsilon$};

\draw[gray] (-.16,2.5) node {$s_\varepsilon$};
\draw[gray] (-1,5.6) node {$r_\varepsilon$};
\end{tikzpicture}
\end{center}
\end{figure}
We consider the following degree-one vortex solution $u_\eps$ of \eqref{eq:euler}:
\[u_\e(x)=f\bigg(\frac{|x-P_\e|}{\e}\bigg) \frac{x-P_\e}{|x-P_\e|} \quad \textrm{for every} \quad x\in \R^2,
\]
where $P_\eps$ is the vortex point (i.e., $u_\eps(P_\eps)=0$), $f:[0,\infty)\to[0,1)$ is the smooth radial profile given by the unique solution of 
\[
-f''-\frac1r f'+\frac1{r^2}f=f(1-f^2) \quad \textrm{for every} \quad r\in (0, \infty),
\]
with $f(0)=0$ and $\lim_{r\to\infty} f(r)=1$; $f$ and $f'$ have the following asymptotics for $r\to\infty$ (see 
\cite{ChenElliottQi:1994a, HerveHerve:1994a})
\[
f(r)=1-\frac1{2r^2}-\frac9{8r^4} +O(r^{-6}), \quad f'(r)= \frac1{r^3}+\frac9{2r^5} +O(r^{-7}).
\]
In a point $x\in \R^2$ with $|x-P_\e|=t$, the Ginzburg-Landau energy density is given by
\[
e_\e(u_\e(x)) = \frac12 \left(\frac{|f'(\frac{t}{\e})|^2}{\e^2}+ \frac{|f(\frac t \eps)|^2}{t^2} \right)  + \frac1{4\e^2} \left(1-|f(\textstyle\frac t \eps)|^2\right)^2,
\]
so that for $t\geq \e$, we find
\be
\label{13}
e_\e(u_\e(x)) 
=    \frac1{2t^2} + \frac1{\e^2}O(\frac{ \e^4}{t^4})
\ee
Recalling that $r_\e\gg \e$, we obtain by integrating over $K_1\setminus B(P_\e, {r_\e})$: 
\[
\int_{K_1\setminus B(P_\e, {r_\e})} e_\e(u_\e) \, dx \le \theta_1\int_{r_\e}^2 t
\left(\frac1{2t^2}  + \frac1{\e^2}O(\frac{ \e^4}{t^4})\right)\, dt
\le  \frac{\theta_1}{2} \log\frac2{r_\e}  +O(\frac{\e^2}{r_\e^2}).
\] 
In $B(P_\e, {r_\e})$, using \eqref{13} and the fact that $f(0)=0$ and $|f'|\lesssim 1$ (in particular, $|f(t)|\lesssim t$ for $t>0$), we estimate 
\begin{align*}
\int_{ B(P_\e, {r_\e})} e_\e(u_\e) \, dx &
  \leq \pi \bigg(\int_0^\eps+\int_\eps^{r_\eps}\bigg) \left[\left(\frac{|f'(\frac{t}{\e})|^2}{\e^2}+ \frac{|f(\frac t \eps)|^2}{t^2} \right)  + \frac1{2\e^2} \left(1-|f(\textstyle\frac t \eps)|^2\right)^2\right] t\, dt \\
&
 \le C\int_0^\eps \frac{t}{\eps^2} \, dt  +\pi \log \frac{r_\e}{\e}+O(1)=\pi \log \frac{r_\e}{\e}+O(1).
\end{align*}
As $\Omega\subset B(P_\e, {r_\e})\cup K_1$, it follows that the interior  energy $M_\eps$ is estimated as:
\[
M_\e = \int_{\Omega} e_\e(u_\e)\, dx \le \pi \log \frac{r_\e}{\e}+\frac{\theta_1}{2} \log\frac2{r_\e}+O(1)\leq\left(\pi (1-\mu) + \frac{\theta_1}{2}\mu\right) |\log\e|+C
\]
where $C>0$ is a constant depending only on $\eta$ and $\theta_0$. Note that for $\mu$ sufficiently close to $1$ and $\e$ small enough, this implies
\[
M_\e \le (\frac{\theta_0} 2+\eta) |\log\e|.
\]
To estimate the boundary energy $N_\eps$, we write $\partial\Omega=\Gamma_1\cup\Gamma_2^+\cup\Gamma_2^-$, where $\Gamma_1$ is the basis of the cone, and $\Gamma_2^\pm$ are the two sides of the cone adjacent to its vertex. Since $P_\e$ is at distance $\sim 1$ of $\Gamma_1$, it holds
\begin{equation*}
\int_{\Gamma_1} e_\e(u_\eps)\, d\mathcal H^1 = O(1).
\end{equation*} 
On the rest of the boundary, note that for every point $x\in\Gamma_2^\pm$ that has a distance $s$ from the orthogonal 
projections of $P_\e$ onto $\Gamma_2^\pm$, we have 
\[
e_\e(u_\e(x)) =  \frac1{2t^2}  + \frac1{\e^2}O(\frac{ \e^4}{t^4}),\qquad\text{where }t=|x-P_\eps|=\sqrt{s^2+d_\e^2},
\]
since $t\geq d_\eps\sim \eps^\mu\gg \eps$.
We can thus estimate
\[
N_\e \le 2 \int_{-\infty}^\infty \left(\frac{1}{2(s^2+d_\e^2)}  +C\frac{\e^2}{(s^2+d_\e^2)^2} \right)\, ds +O(1)
\lesssim \frac1{d_\e}
+\frac{\eps^2}{d_\eps^3}\lesssim\frac1{s_\e}\sim \frac1{\eps^\mu}\ll\frac1{\eps^\alpha}
\]
as $\alpha$ was chosen such that $\alpha=\frac{1+\mu}{2}<1$. 
So \eqref{regime} holds with $\kappa=\frac{\theta_0} 2+\eta$, while $u_\e(P_\e)=0$.\qed

\end{proof}
\begin{rem}
Applying the construction in the proof of Proposition~\ref{p:optimal} to a half-space domain, we deduce that a necessary condition in order that Theorem \ref{thm:main} holds true is given by 
$\kappa\leq \frac\pi 2$ in \eqref{regime} (even for smooth domains $\Omega$). 
\end{rem}

\begin{proof}{ of Proposition~\ref{p:optimal2}}
Let $f:[0,\infty)\to[0,1]$ be a smooth function with $f(0)=0$, $f(r)=1$ for $r\ge 1$ and $|f'(r)|\le C$. Let $x_0\in\dOm$ and 
consider $v_\e(x)=f(\frac{x-x_0}{\e})$ for every $x\in \R^2$. Let $g_\e=(v_\e,0)$ on $\dOm$ and let $u_\e$ be a minimizer of the Ginzburg-Landau energy
with Dirichlet boundary conditions $g_\e$, in particular, $u_\eps(x_0)=g_\eps(x_0)=0$. Then $u_\e$ satisfies \eqref{eq:euler} and (by minimality)
$M_\e\le E_\e(v_\e;\Omega)\lesssim 1$ while $N_\e \lesssim \frac1\e$. \qed
\end{proof}

\section*{Acknowledgements}
We thank Petru Mironescu, Roger Moser and Luc Nguyen for interesting comments. R.I. acknowledges partial support by the ANR project ANR-14-CE25-0009-01.  X.L. acknowledges partial support by the ANR project ANR-18-CE40-0023.

\bibliographystyle{acm}

\bibliography{bibi}

\end{document}